\documentclass[12pt,reqno]{amsart} 

\usepackage{float} 

\usepackage{epsfig,amssymb,amsmath,version}
\usepackage{amssymb,version,graphicx,fancybox,mathrsfs,labelfig}

\usepackage{url,hyperref}
\usepackage{subfigure}
\usepackage{color}
\usepackage{stmaryrd}
\usepackage{multirow}
\usepackage{booktabs,siunitx}
\usepackage{booktabs}
\usepackage{multicol,epstopdf}
\usepackage{xypic}
\usepackage[]{graphicx}
\usepackage[all]{xy}
\usepackage{tikz,ifthen}
\usetikzlibrary{positioning, math, decorations.markings, arrows.meta}
\usetikzlibrary{calc,shapes,cd}
\usetikzlibrary{knots} 
\usepgfmodule{decorations}
\allowdisplaybreaks
\tikzset{knotarrow/.pic={ \draw[edge, <-] (0,0) -- +(-.001,0);}}

\tikzset{edge/.style={line width=0.8}}
\tikzset{wall/.style={very thick}}

\tikzset{->-/.style n args={2}{decoration={markings, mark=at position #1 with {\arrow{#2}}}, postaction={decorate}}} 

\ExplSyntaxOn
\cs_new_eq:NN \ifstreqF \str_if_eq:nnF
\cs_new_eq:NN \ifstreqTF \str_if_eq:nnTF
\ExplSyntaxOff

\tikzset{-o-/.code 2 args={\ifstreqF{#2}{} 
{\ifstreqTF{#2}{>}
   {\pgfkeysalso{decoration={markings,mark=at position #1 with {\arrow[scale=0.8]{#2}}}
                    ,postaction={decorate}}
    }
   {\ifstreqTF{#2}{<}
       {\pgfkeysalso{decoration={markings,mark=at position #1 with {\arrow[scale=0.8]{#2}}}
                    ,postaction={decorate}}
        }
       {\pgfkeysalso{decoration={markings,
                    mark=at position #1 with
                    {\draw[black, fill={#2}] circle[radius=2pt];}}
                    ,postaction={decorate}}
        }
     }
  }}}

\allowdisplaybreaks[4]

\newtheorem{theorem}{Theorem}[section]
\newtheorem{lemma}[theorem]{Lemma}
\newtheorem{definition}[theorem]{Definition}
\newtheorem{corollary}[theorem]{Corollary}
\newtheorem{proposition}[theorem]{Proposition}
\newtheorem{remark}[theorem]{Remark}

\newcommand{\bp}{\begin{proposition}}
\newcommand{\ep}{\end{proposition}}
\newcommand{\bpr}{\begin{proof}}
\newcommand{\epr}{\end{proof}}

\newcommand{\bt}{\begin{theorem}}
\newcommand{\et}{\end{theorem}}

\newcommand{\bl}{\begin{lemma}}
\newcommand{\el}{\end{lemma}}

\newcommand{\bcr}{\begin{corollary}}
\newcommand{\ecr}{\end{corollary}}

\newcommand{\be}{\begin{equation}}
\newcommand{\ee}{\end{equation}}

\newcommand{\bes}{\begin{equation*}}
\newcommand{\ees}{\end{equation*}}

\newcommand{\ba}{\begin{align}}
\newcommand{\ea}{\end{align}}

\newcommand{\bas}{\begin{align*}}
\newcommand{\eas}{\end{align*}}

\graphicspath{{./Figures/} } 

\begin{document}
\bibliographystyle{plain}

\title{On Frobenius algebras obtained from stated skein algebras}
\author{Zhihao Wang}

\keywords{Frobenius algebra, Stated skein algebra, Localization}

 \maketitle


\begin{abstract}

When the quantum parameter $q^{\frac{1}{2}}$ is a root of unity of odd order and the punctured bordered surface has nonempty boundary, we prove the fraction ring of the stated skein algebra (that is the localization over all nonzero elements) is a symmetric Frobenius algebra over both the field of fractions of the image of the Frobenius map and the field of fractions of the center of the stated skein algebra. We also calculate Traces of the fraction ring of the stated skein algebra over these two fields.

%

%
\end{abstract}


\def \cF {\mathcal{F}}
\def \SMQ {\mathscr{S}_{q^{1/2}}(M,\mathcal{N})}
\def \SMQP {\mathscr{S}_{q^{1/2}}(M^{'},\mathcal{N}^{'})}

\def \q {q^{\frac{1}{2}}}

\def \CSM {\mathscr{S}_1(M,\mathcal{N})}

\def \SMN {\mathscr{S}_{q^{1/2}}(M,\mathcal{N})^{(N)}}

\def \S {\mathscr{S}_{q^{1/2}}}

\def \sS {\mathscr{S}}

\def \MN {(M,\mathcal{N})}

\def \sSq {\S(\Sigma)}
\def \sSo {\sS_1(\Sigma)}

\def \bC {\mathbb{C}}

\def \bN {\mathbb{N}}

\def \bZ {\mathbb{Z}}

\def \SZ {Z_{q^{1/2}}(\Sigma)}
\def\sSN {\sSq^{(N)}}
\def \tSZ {\widetilde{{Z_{q^{1/2}}(\Sigma)}}}
\def\tSN {\widetilde{\sSq^{(N)}}}

\section{Introduction}

In this paper, we will work with the complex field $\bC$ with a distinguished nonzero element $\q$, working as the quantum parameter.

The
stated skein  algebra was first defined in \cite{bonahon2011quantum} to establish the quantum trace map, and then it  was refined in \cite{le2018triangular}. It is a generalization of the classical Kauffman bracket skein algebra \cite{przytycki2006skein,turaev1988conway}. The stated skein algebra is useful to understand the Kauffman bracket skein algebra. For example, the quantum trace map for skein algebras can be established using the splitting map for  stated skein algebras \cite{le2018triangular}. Stated skein algebras also have richer structures, such as, the splitting map, and attaching ideal triangles, etc \cite{costantino2022stated1}.

For any pb surface $\Sigma$, please refer to section \ref{stated} for the definition, we use $\sSq$ to denote the stated skein algebra of $\Sigma$.
When $\q$ is a root of unity of odd order, there exists an algebraic embedding 
$\cF:\sSo\rightarrow \sSq$ \cite{bloomquist2020chebyshev,bonahon2016representations},  called the Frbenius map. One reason that makes this map important is because its image is contained in the center  of $\sSq$.

We use $\sSq^{(N)}$ and $Z_{q^{1/2}}(\Sigma)$ to denote $\text{Im}\cF$ and the center of $\sSq$ respectively. 
Set $S=\sSN\setminus\{0\},Q = \SZ\setminus\{0\}$. Then $\sSN[S^{-1}]$ (respectively $\SZ[Q^{-1}]$), denoted as $\tSN$ (respectively $\tSZ$), is a field. The localized stated skein  algebra $\sSq[S^{-1}]$  (respectively $\sSq[Q^{-1}]$) is a finite dimension algebra over the field $\tSN$ (respectively $\tSZ$).

We will show that $\sSq[S^{-1}] = \sSq[Q^{-1}]$ is just the fraction ring of $\sSq$ (that is the localization of $\sSq$ over all nonzero elements), which is denoted as $\text{Fr}(\sSq)$.

Suppose $A$ is a finite dimensional algebra over a field $F$, the dimension of $A$ over $F$ is $k$, and 
$a_1,\cdots,a_k$ is a basis of $A$.
 For any element $a$, we define an $F$-linear map
$L_a:A\rightarrow A, b\mapsto ab$. Suppose $L_a(a_i) = \sum_{1\leq j\leq k}f_{ij} a_j,1\leq i\leq k$.
 Define $\text{Trace}_{F}(a) =\frac{1}{k} \sum_{1\leq i\leq k}f_{ii}\in F$. Then clearly $\text{Trace}_{F}:A\rightarrow F$ is an $F$-linear map, and $\text{Trace}(1_{A})= 1_{F}$.
$\text{Trace}_{F}$ is also symmetric, in a sense that, $\text{Trace}_{F}(ab) = \text{Trace}_{F}(ba)$ for any $a,b\in A$. We say $\text{Trace}_F$ makes $A$ into a symmetric Frobenius algebra over $F$ if, for any nonzero element $a\in A$, there exists an element $b\in A$ such that $\text{Trace}_A(ab)\neq 0$.

\begin{theorem}\label{1}
When every component of $\Sigma$ has nonempty boundary,
 $\text{Trace}_{\tSN}$ (respectively $\text{Trace}_{\tSZ}$) makes $\text{Fr}(\sSq)$ into a symmetric Frobenius algebra over $\tSN$ (respectively $\tSZ$).
\end{theorem}

For a connected pb surface $\Sigma$, we define 
$r(\Sigma) = -\chi(\Sigma)+\sharp\partial\Sigma$, where $\chi(\Sigma)$ is the Euler characteristic of $\Sigma$ and $\sharp\partial\Sigma$ is the number of boundary components of $\Sigma$.

The construction of the quantum trace map, length coordinate version, for stated skein algebras is based on all
peripheral loops $\alpha_1,\cdots,\alpha_t$  (loops going around interior punctures), and some  stated arcs $X_1,\cdots,X_n$, where $t$ is the number of interior punctures and $n=3r(\Sigma)-t$
 \cite{le2022quantum1}. $$\alpha_1,\cdots,\alpha_t,X_1,\cdots,X_n$$ are disjoint from each other, 
$\alpha_1,\cdots,\alpha_t$ lie in $\SZ$, and $X_i X_j = q^{k} X_j X_i$, where $k$ is an integer and $i,j\in\{1,\cdots,n\}$. L{\^e}  and Yu also proved, for any element $\beta\in\sSq$, there exist nonnegative integers $k_1,\cdots,k_r$ such that 
$(X_1)^{k_1}\cdots(X_n)^{k_n} \beta$ is contained in the subalgebra generated by $\alpha_1,\cdots,\alpha_t,X_1,\cdots,X_n$ \cite{le2022quantum1}.

We use
$\{T_n(x)\}_{n\in\mathbb{N}}$ to denote 
 Chebyshev polynomials of the first kind, please refer to section \ref{sec3} for the definition.

\def \wid {\widetilde}

\begin{theorem}\label{2}
Suppose $\Sigma$ is a connected pb surface with nonempty boundary, and $$m_1,\cdots,m_t,k_1,\cdots,k_n$$ are  integers.

(a) \begin{equation}
\begin{split}
\label{Frrr}
&\text{Trace}_{\wid{\sSq^{(N)}}}(T_{m_1}(\alpha_1)\cdots T_{m_t}(\alpha_p) (X_1)^{k_1}\cdots(X_r)^{k_n})\\ =& \left\{ 
    \begin{aligned}
    &T_{m_1}(\alpha_1)\cdots T_{m_t}(\alpha_p) (X_1)^{k_1}\cdots(X_n)^{k_n} & & N\mid m_i,N\mid k_j\text{ for all }1\leq i\leq t,1\leq j\leq n\cr 
    &0 & & \text{otherwise.}
    \end{aligned}
\right.
\end{split}
\end{equation}
$\text{Trace}_{\wid{\sSq^{(N)}}}:\text{Fr}(\sSq)\rightarrow \wid{\sSq^{(N)}}$ as a $\wid{\sSq^{(N)}}$-linear map is unique with the property in equation \eqref{Frrr}.

(b) \begin{equation}
\begin{split}
\label{Frrrr}
&\text{Trace}_{\wid{\SZ}}(T_{m_1}(\alpha_1)\cdots T_{m_t}(\alpha_p) (X_1)^{k_1}\cdots(X_r)^{k_n})\\ =& \left\{ 
    \begin{aligned}
    &T_{m_1}(\alpha_1)\cdots T_{m_t}(\alpha_p) (X_1)^{k_1}\cdots(X_n)^{k_n} & & (k_1,\cdots,k_n)\in\mathcal{B}\cr 
    &0 & & \text{otherwise,}
    \end{aligned}
\right.
\end{split}
\end{equation}
where $\mathcal{B}$ is a subgroup of $\mathbb{Z}^n$ with $(N\mathbb{Z})^n\subset \mathcal{B}$, please refer to equation \eqref{B} for the definition of $\mathcal{B}$.
$\text{Trace}_{\wid{\SZ}}:\text{Fr}(\sSq)\rightarrow \wid{\SZ}$ as a $\wid{\SZ}$-linear map is unique with the property in equation \eqref{Frrrr}.
\end{theorem}

The parallel results for Theorems \ref{1} and \ref{2} when $\partial\Sigma =\emptyset$ were proved in \cite{abdiel2017localized,frohman2016frobenius,frohman2021dimension}.

We use $\text{Hom}_{\text{Alg}}(\sSo,\bC)$ to denote the set of all algebra  homomorphisms from $\sSo$ to $\bC$.
Then $\text{Hom}_{\text{Alg}}(\sSo,\bC)$ is an algebraic variety since $\sSo$ is a commutative domain and is finitely generated as an algebra.
Every $\rho\in \text{Hom}_{\text{Alg}}(\sSo,\bC)$ induces an action of $\sSo$ on $\bC$.
$\sSo$ also has an action on $\sSq$ via the Frobenius map $
\cF$. 
 Then we define 
$\sSq_{\rho} = \sSq\otimes_{\sSo}\bC$.

\begin{theorem}\label{t1}
Suppose $\Sigma$ is connected and $\partial\Sigma\neq \emptyset$.
For every $\rho\in\text{Hom}_{\text{Alg}}(\sSo,\bC)$, $\dim_{\bC}\sSq_{\rho} = N^{3r(\Sigma)}$.
\end{theorem}

\begin{corollary}\label{c}
There exists a proper closed (under Zariski topology) subset $V$ of $\text{Hom}_{\text{Alg}}(\sSo,\bC)$ such that, for any $$\rho\in\text{Hom}_{\text{Alg}}(\sSo,\bC)\setminus V,$$
$\text{Trace}_{\bC}$ makes
$\sSq_{\rho}$ into a Frobenius algebra.
\end{corollary}


{\bf Plan of this paper.}  In section \ref{stated}, we introduce the defintion of stated skein  algebras.
In section \ref{sec3}, we introduce the Forbenius map and the quantum trace map for stated skein algebras.  In section \ref{4}, we prove Theorems \ref{1}, \ref{t1}, and Corollary \ref{c}. In section \ref{5}, we prove Theorem \ref{2}. In section 
\ref{6}, we discuss two special cases: the bigon and the monogon.

{\bf
Acknowledgements}:  The research is supported by NTU  research scholarship, and PhD  scholarship from University of Groningen.

\section{Stated skein modules and algebras}\label{stated}

In this paper we use $\mathbb{Z}$, $\bN$ to denote   the set of integers, the set of nonnegative integers respectively. 
When we say a vector space (respectively an algebra), we mean a vector space (respectively an algebra) over $\bC$.


Let $M$ be an oriented three manifold, and $\mathcal{N}$ be a one dimensional submanifold of $\partial M$  consisting of oriented open intervals such that there is no intersection between the closure of any two open intervals. Then we call $\MN$ a {\bf marked three manifold}.

For a marked three manifold $\MN$, a properly embedded one dimensional submanifold $\alpha$ of $M$  is called an
{\bf $\MN$-tangle} if $\partial\alpha\subset\mathcal{N}$ and $\alpha$ is equipped with a framing  such that  framings at $\partial\alpha$ respect to velocity vectors of $\mathcal{N}$.
If there is a map $s:\partial\alpha\rightarrow \{-,+\}$, then we call $\alpha$ a
 {\bf stated $\MN$-tangle}.

For a marked three manifold $\MN$, 
we use $\text{Tangle}\MN$ to denote the vector space over $\bC$ with all isotopy classes of stated $\MN$-tangles as a basis.
Then the stated skein module $\SMQ$ of $\MN$ is  $\text{Tangle}\MN$ quotient the following relations:
\begin{equation}\label{cross}
\raisebox{-.20in}{
\begin{tikzpicture}
\filldraw[draw=white,fill=gray!20] (-0,-0.2) rectangle (1, 1.2);
\draw [line width =1pt](0.6,0.6)--(1,1);
\draw [line width =1pt](0.6,0.4)--(1,0);
\draw[line width =1pt] (0,0)--(0.4,0.4);
\draw[line width =1pt] (0,1)--(0.4,0.6);
\draw[line width =1pt] (0.6,0.6)--(0.4,0.4);
\end{tikzpicture}
}=
q
\raisebox{-.20in}{
\begin{tikzpicture}
\filldraw[draw=white,fill=gray!20] (-0,-0.2) rectangle (1, 1.2);
\draw [line width =1pt](0.6,0.6)--(1,1);
\draw [line width =1pt](0.6,0.4)--(1,0);
\draw[line width =1pt] (0,0)--(0.4,0.4);
\draw[line width =1pt] (0,1)--(0.4,0.6);
\draw[line width =1pt] (0.6,0.62)--(0.6,0.38);
\draw[line width =1pt] (0.4,0.38)--(0.4,0.62);
\end{tikzpicture}
}
+
 q^{-1}
\raisebox{-.20in}{
\begin{tikzpicture}
\filldraw[draw=white,fill=gray!20] (-0,-0.2) rectangle (1, 1.2);
\draw [line width =1pt](0.6,0.6)--(1,1);
\draw [line width =1pt](0.6,0.4)--(1,0);
\draw[line width =1pt] (0,0)--(0.4,0.4);
\draw[line width =1pt] (0,1)--(0.4,0.6);
\draw[line width =1pt] (0.62,0.6)--(0.38,0.6);
\draw[line width =1pt] (0.62,0.4)--(0.38,0.4);
\end{tikzpicture}
} 
\end{equation}
\begin{equation}\label{unknot}
\raisebox{-.15in}{
\begin{tikzpicture}
\filldraw[draw=white,fill=gray!20] (-0,-0) rectangle (1, 1);
\draw [line width =1pt] (0.5,0.5) circle (0.3);
\end{tikzpicture}
}=-(q^2+q^{-2})
\raisebox{-.15in}{
\begin{tikzpicture}
\filldraw[draw=white,fill=gray!20] (-0,-0) rectangle (1, 1);
\end{tikzpicture}
}
\end{equation}
\begin{equation}\label{arc}
\raisebox{-.26in}{
\begin{tikzpicture}
\filldraw[draw=white,fill=gray!20] (-0,-0) rectangle (1, 1);
\draw [line width =1pt]  (0.5 ,0) arc (-90:225:0.3 and 0.35);
\filldraw[draw=black,fill=black] (0.5,-0) circle (0.09);
\node at (0.3,-0.15) {$-$};
\node at (0.7,-0.15) {\small $+$};
\end{tikzpicture}
}=q^{-\frac{1}{2}}
\raisebox{-.15in}{
\begin{tikzpicture}
\filldraw[draw=white,fill=gray!20] (-0,-0) rectangle (1, 1);
\filldraw[draw=black,fill=black] (0.5,-0) circle (0.09);
\end{tikzpicture}
},\;
\raisebox{-.26in}{
\begin{tikzpicture}
\filldraw[draw=white,fill=gray!20] (-0,-0) rectangle (1, 1);
\draw [line width =1pt]  (0.5 ,0) arc (-90:225:0.3 and 0.35);
\filldraw[draw=black,fill=black] (0.5,-0) circle (0.09);
\node at (0.3,-0.15) {$+$};
\node at (0.7,-0.15) {\small $+$};
\end{tikzpicture}
}=
\raisebox{-.26in}{
\begin{tikzpicture}
\filldraw[draw=white,fill=gray!20] (-0,-0) rectangle (1, 1);
\draw [line width =1pt]  (0.5 ,0) arc (-90:225:0.3 and 0.35);
\filldraw[draw=black,fill=black] (0.5,-0) circle (0.09);
\node at (0.3,-0.15) {$-$};
\node at (0.7,-0.15) {\small $-$};
\end{tikzpicture}
} =0
\end{equation}
\begin{equation}\label{hight}
\raisebox{-.26in}{
\begin{tikzpicture}
\filldraw[draw=white,fill=gray!20] (-0.2,-0) rectangle (1.2, 1);
\draw [line width =1pt](0,1)--(0.4,0.2);
\draw [line width =1pt](1,1)--(0.6,0.2);
\draw [line width =1pt](0.5,0)--(0.6,0.2);
\filldraw[draw=black,fill=black] (0.5,-0) circle (0.09);
\node at (0.3,-0.15) {$+$};
\node at (0.7,-0.15) {\small $-$};
\end{tikzpicture}
}=q^{2}
\raisebox{-.26in}{
\begin{tikzpicture}
\filldraw[draw=white,fill=gray!20] (-0.2,-0) rectangle (1.2, 1);
\draw [line width =1pt](0,1)--(0.4,0.2);
\draw [line width =1pt](1,1)--(0.6,0.2);
\draw [line width =1pt](0.5,0)--(0.6,0.2);
\filldraw[draw=black,fill=black] (0.5,-0) circle (0.09);
\node at (0.3,-0.15) {$-$};
\node at (0.7,-0.15) {\small $+$};
\end{tikzpicture}
} + q^{-\frac{1}{2}}
\raisebox{-.16in}{
\begin{tikzpicture}
\filldraw[draw=white,fill=gray!20] (-0.2,-0) rectangle (1.2, 1);
\draw [line width =1pt](0,1)--(0.4,0.2);
\draw [line width =1pt](1,1)--(0.6,0.2);
\draw [line width =1pt](0.38,0.2)--(0.62,0.2);
\filldraw[draw=black,fill=black] (0.5,-0) circle (0.09);
\end{tikzpicture}
}
\end{equation}
\def \cN {\mathcal{N}}
where each black dot represents an oriented interval in $\cN$ with the orientation pointing towards readers, each gray square represents a projection of an embedded cube in $M$, the black lines are parts of stated $\MN$-tangles, and in each equation the parts of $\MN$-tangles outside the gray squares are identical.  Please refer to \cite{bloomquist2020chebyshev, le2018triangular} for detailed explanation.

If $\cN=\emptyset$, there are only equations \eqref{cross} and \eqref{unknot}. In this case the stated skein module is the (Kauffman bracket) skein module \cite{przytycki2006skein,turaev1988conway}.

Let $\overline{\Sigma}$ be an oriented compact surface. A {\bf punctured bordered surface} (or  {\bf pb surface}) $\Sigma$ is obtained from $\overline{\Sigma}$ by removing finite points, which are called punctures, such that every boundary component of $\overline{\Sigma}$ contains at least one puncture. The punctures, contained in the interior of $\overline{\Sigma}$, are called {\bf interior punctures}.

The {\bf bigon} is obtained from the 2-dimensional disk by removing two punctures on the boundary. The {\bf monogon} is obtained from the 2-dimensional disk by removing one puncture on the boundary. 

For a pb surface $\Sigma$, we select one point on each boundary component of $\Sigma$, let $P$ be the union of all these selected points. Define $M=\Sigma\times [0,1],\cN = P\times (0,1)$ such that the oriention of $\cN$ is given by the positive orientation of $(0,1)$, then $\MN$ is a marked three manifold.
We use $\S(\Sigma)$ to denote $\SMQ$.
Then $\S(\Sigma)$ has an algebra structure given by staking stated skein tangles, that is, for any two stated skein tangles $\alpha,\beta$, $\alpha\beta$ is defined to be staking $\alpha$ above $\beta$. We call $\S(\Sigma)$ the stated skein algebra of $\Sigma$. We will use $Z_{q^{1/2}}(\Sigma)$ to denote the center of $\sSq$.
When $\partial \Sigma=\emptyset$, $\S(\Sigma)$ is the (Kauffman bracket) skein algebra.

For any two pb surfaces $\Sigma_1,\Sigma_2$, obviously we have $\S(\Sigma_1\cup\Sigma_2)\simeq \S(\Sigma_1)\otimes\S(\Sigma_2)$ as algebras. From now on, we  assume all the pb surfaces, involved in this paper, are connected.

We call the orientation of $\partial\Sigma$, induced by the orientation of $\Sigma$, the {\bf positive orientation} of $\partial\Sigma$, and call the orientation of $\partial\Sigma$ opposite to the positive orientation as {\bf negative orientation.}

Let $h:\Sigma_1\rightarrow \Sigma_2$ be a proper embedding for two pb surfaces $\Sigma_1,\Sigma_2$.
It is possible that $h$ maps more than one boundary component of $\Sigma_1$ into one boundary component of  
$\Sigma_2$. For each boundary component $b$ of $\Sigma_2$, we give a linear order on the set of boundary components of $\Sigma_1$ that are mapped into $b$ by $h$, and call $h$ an {\bf ordered proper embedding} from 
$\Sigma_1$ to $\Sigma_2$. If the linear order, corresponding to each boundary component $b$ of $\Sigma_2$, is induced by the positive (respectively negative) orientation of $\partial\Sigma_2$, we call $h$ {\bf positively ordered} (respectively {\bf negatively ordered}). An ordered proper embedding $h:\Sigma_1\rightarrow \Sigma_2$ induces a linear map $h_{*}:\S(\Sigma_1)\rightarrow \S(\Sigma_2)$ \cite{costantino2022stated1}.

\section{Frobenius map and quantum trace map}\label{sec3}

In order to define Frobenius map, first we have to introduce Chebyshev polynomials, which are defined by the following 
 recurrenc relation:
\begin{equation}\label{cheby}
P_k(x) = xP_{k-1}(x) - P_{k-2}(x).
\end{equation}
We set $T_0(x) = 2, T_1(x) = x$, using the recurrenc relation \eqref{cheby}, we get a sequence of polynomials
$\{T_n(x)\}_{n\in\mathbb{N}}$, which are called
 Chebyshev polynomials of the first kind.

\subsection{Frobenius map}
Let $(M,\mathcal{N})$ be a marked three manifold, and $\alpha$ be a framed knot or stated framed arc in
$\MN$. For any $n\in \bN$, we use $\alpha^{(n)}$ to denote a new stated $\MN$-tangle obtained   by threading $\alpha$ to $n$ parallel copies along the framing direction.
For a polynomial $Q(x)=\sum_{0\leq t\leq n} k_t x^t$,  define
$$Q(\alpha) = \sum_{0\leq t\leq n} k_t \alpha^{(t)}\in\SMQ.$$

\def \cF {\mathcal{F}}

Suppose $\q$ is a root of unity of odd order $N$, then there is a linear map 
$\cF:\sS_{1}\MN\rightarrow \SMQ,$
called Frobenius map \cite{bloomquist2020chebyshev,bonahon2016representations}.
 Let $\alpha$ be any stated $(M,\mathcal{N})$-tangle, suppose 
$\alpha = K_1\cup\cdots\cup K_m \cup C_1\cup\cdots\cup C_n$
where $K_i,1\leq i\leq m,$ are framed knots and $C_j,1\leq j\leq n,$ are stated framed arcs.
Then 
\begin{equation}\label{Frobe}
\cF(\alpha) =T_N(K_1)\cup\cdots\cup T_N(K_m)\cup C_1^{(N)}\cup\cdots\cup C_n^{(N)}.
\end{equation}

\def \Im {\text{Im}}

\def \Zq {Z_{q^{1/2}}(\Sigma)}

Let $\Sigma$ be a pb surface, then $\cF:\sS_1(\Sigma)\rightarrow \sSq$ is injective and 
$\Im \cF$ is contained in $\Zq$ \cite{bloomquist2020chebyshev,bonahon2016representations}.

\def \otau {\overline{\tau}}
\def \Trq {\mathcal{T}_{q^{1/2},\tau}(\Sigma)}
\def \Trp {\mathcal{T}_{q^{1/2},\tau}^{+}(\Sigma)}
\def \Tro {\mathcal{T}_{1,\tau}(\Sigma)}
\def \Tr {\text{Tr}}

\subsection{Quantum trace map}

Let $A$ be any algebra over $\bC$. Two elements $a,b\in A$ are said to be {\bf $q$-commuting}, if
$ab = q^{C(a,b)}ba$ for some integer $C(a,b)$.
Suppose $x_1,\cdots,x_k\in A$ and any two of them are $q$-commuting to each other. Then 
define
 $$[x_{1}x_{2}\dots x_{k}] = q^{-\frac{1}{2}\sum_{1\leq j<l\leq k}C(x_{j},x_{l})}x_{1}x_{2}\dots x_{k}.$$
Then it is easy to show $[x_{1}x_{2}\dots x_{k}]$ is independent of the order of $x_{1}x_{2}\dots x_{k}$.

Let $\Sigma$ be a pb surface with nonempty boundary, and  $\Sigma$ is not a bigon or a monogon.
An {\bf ideal arc} is an embedding $c:(0,1)\rightarrow \Sigma$ such that
$c$ extends to a properly embedding $\bar{c}:[0,1]\rightarrow \overline{\Sigma}$ with 
$\bar{c}(0),\bar{c}(1)$ being punctures.
The ideal arc that is isotopic to a boundary component of $\Sigma$ is called a {\bf boundary ideal arc}.
 Then a {\bf quasitriangulation} $\tau$ is a collection of non-trivial ideal arcs such that 
(1) no two arcs in $\tau$ are isotopic, (2) $\tau$ is maximal under condition (1).
Note that $\tau$ contains all boundary ideal arcs. We use $\tau_{\partial}$ to denote the set of all boundary ideal arcs.

Let $\overline{\tau_{\partial}}$ be another copy of $\tau_{\partial}$, that is, $\overline{\tau_{\partial}} = 
\{\bar{e}\mid e\in \tau_{\partial}\}$. Then define $\overline{\tau}
=\tau\cup \overline{\tau_{\partial}}$. For any two elements $a,b\in \overline{\tau}$, there is a defined
integer $\bar{P}(a,b)$, see subsection 4.2 in
\cite{le2022quantum1} for the definition of $\bar{P}(a,b)$.

For every interior puncture $p$, we use $\alpha_p$ to denote the peripheral loop around $p$ with vertical framing.
Let $\mathcal{P}$ be the  set of all interior punctures of $\Sigma$. Then define 
$$\Trq = \bC[\alpha_p\mid p\in \mathcal{P}]\langle  x_a^{\pm}\mid a\in\overline{\tau}\rangle/ (x_a x_b = q^{\bar{P}(a,b)} x_b x_a).$$
We use $\mathbb{Z}^{\otau}$ to denote the set of all maps from $\otau$ to $\mathbb{Z}$, then $\mathbb{Z}^{\otau}$ has an obvious abelian group structure, induced by the addition operation of $\bZ$.
For any $\vec{k}$, we define $x^{\vec{k}} = [\prod_{e\in\otau} x_e^{\vec{k}(e)}]$.
We define $\bN^{\otau}$ to be a subset of $\mathbb{Z}^{\otau}$ consisting of all maps from 
$\otau$ to $\mathbb{N}$.

For any element $e\in\otau$, we try to define an element $X_e\in \sSq$.
We slightly move the two endpoints of $e$ along $\partial\Sigma$ in the negative orientation, then give vertical framing to $e$, which is denoted as $e^{'}$. If the two endpoints of $e$ are distinct, $e^{'}$ alreday is a tangle diagram in $\Sigma$, and we use $T(e)$ to denote $e^{'}$. If the two endpoints of $e$ are the same, then the two endpoints of $e^{'}$ are also the same and are contained in a boundary component $c$. 
Then we split $e^{'}$ at the common endpoint, which is denoted as $e^{''}$, such that the two endpoints of $e^{''}$ still contained in $c$ and there is no newly created crossings in $e^{''}$. 
We require the endpoint of $e^{''}$, which we meet first when we walk along $c$ in it's positive direction (the direction induced by the orientation of $\Sigma$), is lower than the other one.
Then we get a tangle diagram in $\Sigma$, which is also denoted as $T(e).$

When $e\in\tau$, we  give 
both the two endpoints the state $+$, which is denoted as $T(e)(+,+)$.
Then we define $X_e = T(e)\in \sSq$ if two endpoints of $e$ are distinct, and 
define $X_e =q^{-1/2} T(e)\in\sSq$ otherwise. 

When $e\in\overline{\partial\Sigma}$, we state one endpoint of $T(e)$ with $-$ and another one with $+$, which is denoted as $T(e)(+,-)$.
If two endpoints of $e$ are distinct, the endpoint of $T(e)$ contained in $e$ is stated with $-$. 
Define $X_e = T(e)(+,-)$.
If two endpoints of $e$ are the same, then two endpoints of $T(e)$ are both contained in $e$.
Then endpoint of $T(e)$, which we meet first when we walk along $e$ in it's positive direction, is stated with $-$. Define $X_e = q^{1/2} T(e)(+,-)$.

From \cite{le2022quantum1}, we know
for any $a,b\in\otau$, we have 
\begin{equation}\label{qcom}
X_a X_b = q^{\bar{P}(a,b)}X_bX_a.
\end{equation}
 Then
for any $\vec{k}\in\bZ^{\otau}$, we can define $X^{\vec{k}} = [\prod_{e\in\otau}(X_e)^{\vec{k}(e)}]\in\sSq.$

Let $\Trp$ be a $\bC[\alpha_p\mid p\in \mathcal{P}]$-subalgebra of $\Trq$ generated by $x_e,e\in\otau$.

\def \Sqp {\mathscr{S}_{q^{1/2}}^{+}(\Sigma)}
\def \Sqo {\mathscr{S}_{1}^{+}(\Sigma)}

Note that $\bC[\alpha_p\mid p\in \mathcal{P}]\subset \Zq$, thus $\sSq$ is also a $\bC[\alpha_p\mid p\in \mathcal{P}]$-algebra.
Then equation \eqref{qcom} indicates there is a $\bC[\alpha_p\mid p\in \mathcal{P}]$
algebra homomorphism $\iota : \Trp\rightarrow \sSq$ defined by $\iota(x_e) = X_e$ for $e\in\otau$. $\iota$ is actually injective, and induces a $\bC[\alpha_p\mid p\in \mathcal{P}]$-algebraic embedding 
$\Tr :\sSq\rightarrow \Trq$ such that $\Tr\circ \iota = Id_{\Trp}$ \cite{le2022quantum1}.

We can regard $\sSq$ as a $\bC[\alpha_p\mid p\in \mathcal{P}]$-subalgebra of $\Trq$ using the embedding $\Tr$.
Then we have  
\begin{equation}\label{subset}
\Trp\subset\sSq\subset \Trq.
\end{equation}
For any element $l\in\sSq$, there exists $\vec{u}\in\bN^{\otau}$ such that $x^{\vec{u}} l\in \Trp$ \cite{le2022quantum1}.

It is well-know that $\Trp,\sSq,\Trq$ all are Ore domain. Let $Z$ be the set of nonzero elements of
$\Trp$, the define $\text{Fr}(\Trp) = \Trp[Z^{-1}]$. Similarly we define 
$\text{Fr}(\sSq),\text{Fr}(\Trq)$. Then equation \eqref{subset} indicates the following Lemma. 

\begin{lemma}(\cite{le2022quantum1}) \label{equal}
$\Tr :\sSq\rightarrow \Trq$ induces the isomorphism 
$$\widetilde{\Tr}: \text{Fr}(\sSq)\rightarrow \text{Fr}(\Trq).$$
If we regard $\sSq$ as a $\bC[\alpha_p\mid p\in \mathcal{P}]$-subalgebra of $\Trq$, then we have
$$
\text{Fr}(\Trp)=\text{Fr}(\sSq)=\text{Fr}(\Trq).$$
\end{lemma}

\def \Nta {\bN^{\otau}}
\def \Zta {\bZ^{\otau}}

%
%

\subsection{Compatibility between Frobenius map and Quantum trace map}

Let $\Sigma$ be a pb surface with nonempty boundary, and let $\tau$ be a quasitriangulation. We assume $\q$ is a root of unity of odd order $N$.
%
%
%

Clearly there exists an algebraic embedding
$F:\Tro\rightarrow \Trq$, defined by $F(\alpha_p) = T_N(\alpha_p),p\in \mathcal{P}, 
F(x_e) = (x_e)^N,e\in\otau$.

\begin{proposition}(Theorem 5.2 in \cite{bloomquist2020chebyshev})\label{commu}
The following diagram is commutative:
$$
\begin{tikzcd}
\sSo  \arrow[r, "\cF"]
\arrow[d, "\Tr"]  
&  \sSq  \arrow[d, "\Tr"] \\
 \Tro  \arrow[r, "F"] 
&  \Trq\\
\end{tikzcd}.
$$
\end{proposition}
%
%
%
%

\def \SZ {Z_{q^{1/2}}(\Sigma)}
\def\sSN {\sSq^{(N)}}
\def \tSZ {\widetilde{{Z_{q^{1/2}}(\Sigma)}}}
\def\tSN {\widetilde{\sSq^{(N)}}}

\section{Frobenius algebras  obtained from the stated skein algebra}\label{4}
In the following of this paper, unless especially specified, we will assume all the pb surfaces have nonempty boundary, and $\q$ is a root of unity of odd order $N$.

\subsection{Frobenius algebras obtained by localization}
Let $\Sigma$ be a pb surface. Recall that we use $\SZ$ to denote the center of $\sSq$, which was fully described in \cite{yu2023center}. Then $\SZ$ is a commutative domain.
We know there is an algebraic embedding $\cF:\sSo\rightarrow \sSq$ such that 
$\Im\cF\subset Z_{q^{1/2}}(\Sigma)$.
We use $\sSq^{(N)}$ to denote $\Im\cF$, then $\sSN$ is also a commutative domain. 

Let $S=\sSN\setminus \{0\}, Q=\SZ\setminus \{0\}$. Define 
$$\tSN = \sSN[S^{-1}], \tSZ = \SZ[Q^{-1}]. $$
Then  $\sSq[S^{-1}]$ is a finite dimensional algebra over the   field $\tSN$, and the dimension was calculated in \cite{wang2023finiteness}. Similarly, $\sSq[Q^{-1}]$ is a finite dimensional algebra over the   field $\tSZ$, and the dimension was calculated in \cite{yu2023center}. 

\begin{definition}
Let $F$ be a field, and $A$ be an algebra over $F$. $A$ is called a division algebra over $F$ if any nonzero element in $A$ is invertible.
\end{definition}

\begin{definition}
Let $F$ be a field, and $A$ be an algebra over $F$. $A$ is called a symmetric Frobenius algebra over $F$ if there is an $F$-linear map $\varepsilon:A\rightarrow F$ such that (1) 
$\varepsilon(ab) = \varepsilon(ba)$ for any $a,b\in A$ and (2) $\varepsilon(ab) = 0$ for all $b\in A$ indicates $a=0$.
\end{definition}

\def \T {\text{Trace}}

Suppose $A$ is a finite dimensional algebra over $F$, the dimension of $A$ over $F$ is $k$, and 
$a_1,\cdots,a_k$ is a basis of $A$. Recall that,
 for any element $a$, we define an $F$-linear map
$L_a:A\rightarrow A, b\mapsto ab$. Suppose $L_a(a_i) = \sum_{1\leq j\leq k}f_{ij} a_j,1\leq i\leq k$.
 Define $\text{Trace}_{F}(a) =\frac{1}{k} \sum_{1\leq i\leq k}f_{ii}\in F$. 

\begin{lemma}\label{division}(\cite{frohman2021dimension})
Let $F$ be a field, and $A$ be an algebra over $F$.

(1) If $A$ is a domain and is  finite dimensional  over $F$, then $A$ is a division algebra.

(2) Suppose $A$ is a finite dimensional division algebra over $F$. Then $\T_F$ makes $A$ into a symmetric Frobenius algebra over $F$.
\end{lemma}

\bt\label{Frobenius}
Let $\Sigma$ be a pb surface. Then we have 

(a)  $\sSq[S^{-1}]$ is  a division algebra over $\tSN$, and $\sSq[Q^{-1}]$ is  a division algebra over $\tSZ$.

(b)  $\T_{\tSN}$  
 makes $\sSq[S^{-1}]$ into a Frobenius algebra over $\tSN$, and 
$\T_{\tSZ}$ makes $\sSq[Q^{-1}]$ into a Frobenius algebra over $\tSZ$.
\et
\begin{proof}
Since $\sSq$ is a domain, then both $\sSq[S^{-1}]$ and $\sSq[Q^{-1}]$ are domain.
We also have $\sSq[S^{-1}]$ is a finite dimensional algebra over the   field $\tSN$, and $\sSq[Q^{-1}]$ is a finite dimensional algebra over the   field $\tSZ$. Then (a) is implied by Lemma \ref{division}.
Obviously, (a) and Lemma \ref{division} imply (b).
\end{proof}


\def \Max {\text{MaxSpec}(\sSo)}

\subsection{Frobenius algebras obtained by representation varieties}
The set of maximal ideals of $\sSo$, which will be denoted as $\text{MaxSpec}(\sSo)$, is an algebraic variety.
$\text{MaxSpec}(\sSo)$ is actually isomorphic to the representation variety of a fundamental groupoid, associated to $\Sigma$ \cite{costantino2022stated1,wang2023stated}.
Note that an element $\rho\in\Max$ corresponds to an algebra homomorphism from 
$\sSo$ to $\bC$. Meanwhile  an  algebra homomorphism from 
$\sSo$ to $\bC$ corresponds to an element in $\Max$. Since this correspondence is a bijection, we will not distinguish between an  algebra homomorphism from 
$\sSo$ to $\bC$ and an element in $\Max$.
 For any element $\rho\in\Max$,
$\sSo$ has an action on $\bC$ induced by $\rho,$ that is, for $x\in\sSo,k\in \bC$, $x\cdot k = \rho(x)k$.
$\sSo$ also has an action on $\sSq$ induced by $\cF,$ that is, for $x\in\sSo,y\in \sSq$, $x\cdot y = \cF(x)y$.
Then we define $\sSq_{\rho} = \sSq\otimes_{\sSo}^{\rho}\bC$, where the superscript is to indicate the action of 
$\sSo$ on $\bC$ is induced by $\rho$.

A {\bf saturated system} of $\Sigma$ is a collection of disjoint properly embedded arcs $$\{c_1,\cdots,c_r\}$$ in $\Sigma$ such that (1) each component of $\Sigma\setminus \cup_{1\leq i\leq r}c_i$ contains exactly one puncture,
(2) $\{c_1,\cdots,c_r\}$ is maximal under condition (1). For each $1\leq i\leq r$, we use 
$U(c_i)$ to denote a small enough open tubular neighborhood of $c_i$ such that $U(c_i)$ is isomorphic to $c_i\times (0,1)$
and $\partial c_i\times (0,1)\subset\partial \Sigma$.  Then $U(c_i)$ is isomorphic to a bigon. 
Let $L:\cup_{1\leq i\leq r}U(c_i)\rightarrow \Sigma$ be the embedding, and suppose $L$ is negatively ordered. 

Recall that, for a pb surface $\Sigma$, we define 
$r(\Sigma) = -\chi(\Sigma)+\sharp\partial\Sigma$, where $\chi(\Sigma)$ is the Euler characteristic of $\Sigma$ and $\sharp\partial\Sigma$ is the number of boundary components of $\Sigma$.

\begin{lemma}(Theorem 7.13 in \cite{le2021stated})\label{lll}
(a) $r=r(\Sigma)$.

(b) $L$ induced a $\bC$-linear isomorphism $$L_{*}:\S(\cup_{1\leq i\leq r} U(c_i))\rightarrow \S(\Sigma).$$
\end{lemma}

\begin{lemma}\label{Le}
The following diagram is commutative:
\begin{equation}\label{eq_com}
\begin{tikzcd}
\sS_1(\cup_{1\leq i\leq r} U(c_i))  \arrow[r, "L_{*}"]
\arrow[d, "\cF"]  
&  \sS_1(\Sigma) \arrow[d, "\cF"] \\
 \S(\cup_{1\leq i\leq r} U(c_i))  \arrow[r, "L_{*}"] 
&  \S(\Sigma)\\
\end{tikzcd}.
\end{equation}
Especially, the isomorphism $L_{*}:\S(\cup_{1\leq i\leq r} U(c_i))\rightarrow \S(\Sigma)$ preserves module structures over $\sS_1(\cup_{1\leq i\leq r} U(c_i))$ and $\sS_1(\Sigma)$, that is, for any $x\in \sS_1(\cup_{1\leq i\leq r} U(c_i)), y\in \S(\cup_{1\leq i\leq r} U(c_i))$, we have 
$L_*(x\cdot y) = L_*(x)\cdot L_*(y).$
\end{lemma}
\begin{proof}
This Lemma is a special case (when $n=2$) for Lemma 8.4 in \cite{wang2023stated}.
\end{proof}

\def \Oq {O_q(SL_2)}
\def \O {O(SL_2)}

The algebra
$\Oq$ is generated by $a,b,c,d$ and subject to the following relations:
\begin{align*}
ca &= q^2ac, db = q^2bd, ba = q^2ab, dc=q^2cd,\\
bc&=cb, ad - q^{-2}bc =1, da-q^2 cb =1.
\end{align*}
We  use $\O$ to denote the classical case of $\Oq$ (that is to denote  $O_1(SL_2)).$
For each $1\leq i\leq r$, $U(c_i)$ is bigon, then $\S(U(c_i)) = \Oq$ \cite{costantino2022stated1}.

\def \SL {SL_2(\bC)}

\def \M {\text{MaxSpec}}
Note that $L_*: \sS_1(\cup_{1\leq i\leq r} U(c_i))\rightarrow\sS_1(\Sigma)$ is an algebraic isomorphism.
Thus $\M(\O^{\otimes r}) = \M(\sSo)$. We know that $\M(\O) = SL_2(\bC)$. For any element 
$\begin{pmatrix}
k_{11} & k_{12}\\
k_{21} & k_{22}
\end{pmatrix}\in SL_2(\bC),$ we can regard $\begin{pmatrix}
k_{11} & k_{12}\\
k_{21} & k_{22}
\end{pmatrix}$ as an element in $\M(\O)$ by 
$$
\begin{pmatrix}
k_{11} & k_{12}\\
k_{21} & k_{22}
\end{pmatrix} (a) = k_{11},
\begin{pmatrix}
k_{11} & k_{12}\\
k_{21} & k_{22}
\end{pmatrix} (b) = k_{12},
\begin{pmatrix}
k_{11} & k_{12}\\
k_{21} & k_{22}
\end{pmatrix} (c) = k_{21},
\begin{pmatrix}
k_{11} & k_{12}\\
k_{21} & k_{22}
\end{pmatrix} (d) = k_{22}.
$$
Then 
\begin{equation}\label{ident}
\M(\sSo)=\M((\O)^{\otimes r}) = \SL\times \cdots \times \SL,
\end{equation}
where the number of copies of $\SL$ is $r$.
The Frobenius map, $\cF:\O\rightarrow \Oq$, for $\Oq$ (or the bigon) is defined as following:
\begin{equation}\label{eq68}
\cF(a) = a^N,\cF(b) = b^N,\cF(c) = c^N,\cF(d) = d^N.
\end{equation}
Every element $\rho\in\M(\O)$ induces an action of $\O$ on $\bC$, $\O$ also has an action on $\Oq$ induced by $\cF$, then 
define $\Oq_{\rho} = \Oq{\otimes}_{\O}^{\rho}\bC$.

\def \dim {\text{dim}}
\begin{lemma}(\cite{brown2012lectures})\label{O}
For any $\rho\in\M(\O)$, we have $\dim_{\bC}\Oq_{\rho}  = N^{3}$.
\end{lemma}

Remark 6.26 in \cite{wang2023finiteness} indicates the following Lemma.
\begin{lemma}\label{bas}
For $\rho = \begin{pmatrix}
k_{11} & k_{12}\\
k_{21} & k_{22}
\end{pmatrix}\in\M(\O)$, if $k_{11}\neq 0$ (respectively $k_{22}\neq 0$), 
$\{a^{s_1}b^{s_2}c^{s_3}\mid s_1,s_2,s_3\in\mathbb{N}, 0\leq s_1,s_2,s_3\leq N-1\}$
(respectively $\{d^{s_1}b^{s_2}c^{s_3}\mid s_1,s_2,s_3\in\mathbb{N}, 0\leq s_1,s_2,s_3\leq N-1\}$) projected to a basis of 
$\Oq_{\rho}$.
\end{lemma}

Any element $\rho\in\M(\O^{\otimes r})$ is of the form that $\rho = \rho_1\otimes\cdots \otimes \rho_r$, where 
$\rho_i$ is an algebra homomorphism from $\O$ to $\bC$ for each $1\leq i\leq r$. Then we have the following Lemma.

\begin{lemma}\label{is}
$\Oq ^{\otimes r}\otimes _{\O^{\otimes r}}^{\rho} \bC \simeq 
(\Oq \otimes _{\O}^{\rho_1} \bC)\otimes \cdots\otimes( \Oq \otimes _{\O}^{\rho_r} \bC)$ as algebras.
\end{lemma}
\begin{proof}
It is easy to show the following two maps are well-defined algebraic maps:
\begin{align*}
\nu:\Oq ^{\otimes r}\otimes _{\O^{\otimes r}}^{\rho} \bC &\rightarrow 
(\Oq \otimes _{\O}^{\rho_1} \bC)\otimes \cdots\otimes( \Oq \otimes _{\O}^{\rho_r} \bC)\\
(x_1\otimes \cdots\otimes x_r)\otimes_{\O^{\otimes r}}^{\rho} 1&\mapsto (x_1 \otimes _{\O}^{\rho_1} 1)\otimes \cdots\otimes( x_r \otimes _{\O}^{\rho_r} 1),
\end{align*}
\begin{align*}
\mu:(\Oq \otimes _{\O}^{\rho_1} \bC)\otimes \cdots\otimes( \Oq \otimes _{\O}^{\rho_r} \bC)&\rightarrow
\Oq ^{\otimes r}\otimes _{\O^{\otimes r}}^{\rho} \bC \\
(x_1 \otimes _{\O}^{\rho_1} 1)\otimes \cdots\otimes( x_r \otimes _{\O}^{\rho_r} 1)&\mapsto
(x_1\otimes \cdots\otimes x_r)\otimes_{\O^{\otimes r}}^{\rho} 1.
\end{align*}
Obviously, $\nu$ and $\mu$ are inverse to each other.
\end{proof}

\begin{theorem}
For every $\rho\in\Max$, $\dim_{\bC}\sSq_{\rho} = N^{3r(\Sigma)}$.
\end{theorem}
\begin{proof}
Lemmas \ref{lll}, \ref{Le}, \ref{O}, \ref{is}, and equation \ref{ident}.
\end{proof}

For each $1\leq i\leq r$, define
\begin{align*}
W_i = \{&\rho = \rho_1\otimes\cdots \otimes \rho_r\in \M(\O^{\otimes r})\mid \\&\rho_t\in \text{Hom}_{\text{Alg}}(\O,\bC), 1\leq t\leq r,
\rho_i(a) = \rho_i(d) = 0\}.
\end{align*}
Then $W_i$ is a closed subset of $\M(\O^{\otimes r})$ (under Zariski topology) for each $1\leq i\leq r$.
Define $$W=\cup_{1\leq i\leq r}W_i,$$ then $W$ is  a proper closed subset of $\M(\O^{\otimes r})$.

\begin{lemma}\label{68}
With the identification in equation \eqref{ident}, for every $$\rho\in\Max\setminus W,$$
there exists a basis of $\sSq[S^{-1}]$ over $\widetilde{\sSN}$ which are projected to a basis for 
$\sSq_{\rho}.$
\end{lemma}
\begin{proof}
Suppose $\rho = \rho_1\otimes\cdots \otimes \rho_r$, where $\rho_i$ is an algebra homomorphism from $\O$ to $\bC$ for each $1\leq i\leq r$. Since $\rho\notin W$,
for each $1\leq i\leq r$, we have $\rho_i(a)\neq 0$ or $\rho_i(d)\neq 0$. If $\rho_i(a)\neq 0$, define 
$$X_{i} = \{a^{s_1}b^{s_2}c^{s_3}\mid s_1,s_2,s_3\in\mathbb{N}, 0\leq s_1,s_2,s_3\leq N-1\}.$$
Otherwise, define 
$$X_i = \{d^{s_1}b^{s_2}c^{s_3}\mid s_1,s_2,s_3\in\mathbb{N}, 0\leq s_1,s_2,s_3\leq N-1\}.$$
Lemmas \ref{Le}, \ref{bas}, \ref{is} indicate 
$$\{L_{*}(x_1\otimes \cdots\otimes x_r)\mid x_i\in X_i,1\leq i\leq r\}$$
projected to a basis of $\sSq_{\rho}$.
Remarks 6.26 and 6.34 in \cite{wang2023finiteness} indicate $\{L_{*}(x_1\otimes \cdots\otimes x_r)\mid x_i\in X_i,1\leq i\leq r\}$ is a basis of $\sSq[S^{-1}]$ over $\widetilde{\sSN}$.
\end{proof}

\begin{corollary}
There exists a proper closed subset $V$ of $\Max$ such that, for any $$\rho\in\Max\setminus V,$$
$\text{Trace}_{\bC}$ makes
$\sSq_{\rho}$ into a Frobenius algebra.
\end{corollary}
\begin{proof}
(b) in Theorem \ref{Frobenius},
Lemma \ref{68},  Remark 6.26  in \cite{wang2023finiteness}, and the fact that $\Max $ is irreducible.
\end{proof}

\section{Calculation for Trace}\label{5}
Unless especially specified,
in this section, we will assume $\Sigma$ is a connected pb surface with nonempty boundary, $\Sigma$ is not a bigon or monogon, $\tau$ is a quasitriangulation of $\Sigma$.

\def \sSS{\sSq[S^{-1}]}
\def \sSQ{\sSq[Q^{-1}]}

\def \Frs {\text{Fr}(\sSq)}

\begin{lemma}\label{equali}
$\sSS=\sSQ=\text{Fr}(\sSq).$
\end{lemma}
\begin{proof}
(a) in
Theorem  \ref{Frobenius} indicates every nonzero element in $\sSS$ is invertible, thus $\sSS=\Frs$. Similarly we have $\sSQ=\text{Fr}(\sSq)$.
\end{proof}

\subsection{On $\T_{\tSN}$}
Recall that
$r(\Sigma) = -\chi({\Sigma})+\sharp\partial\Sigma$ where 
$\chi({\Sigma})$ is the Euler characteristic of ${\Sigma}$ and $\sharp\partial\Sigma$ is the number of boundary components of $\Sigma$, 
$\mathcal{P}$ is the set of interior punctures of $\Sigma$.
 We use $P$ to denote $|\mathcal{P}|$. Then using Euler characteristic, it is easy to show $|\otau| =3 r(\Sigma) - P$.

\def \bNP {\bN^{\mathcal{P}}}
\def \bZT {\bZ^{\otau}}

Recall that, for each interior puncture $p$, we use $\alpha_p$ to denote the peripheral loop around $p$ with vertical framing.

We use $\bN^{\mathcal{P}}$ to denote the set of all maps from $\mathcal{P}$ to $\bN$, and use
$\bN^{\mathcal{P}}_{<N}$ to denote $\{\vec{a}\in\bNP\mid 0\leq \vec{a}(p)\leq N-1,p\in\mathcal{P}  \}$.
We use $\vec{0}$ to denote the map in $\bN^{\mathcal{P}}$ such that $\vec{0}(p) = 0$ for all $p\in\mathcal{P}$.
 For any element 
$\vec{a}\in \bNP$, we define $\alpha^{\vec{a}} = \prod_{p\in\mathcal{P}}T_{\vec{a}(p)}(\alpha_p)$ and ${}^{\vec{a}}\alpha = \prod_{p\in\mathcal{P}} (\alpha_p)^{\vec{a}(
p)}$.
Recall that $\bZT$ denotes the set of all maps from $\otau$ to $\bZ$.
We also use $\vec{0}$ to denote the map in $\bZT$ such that $\vec{0}(e) = 0$ for all $e\in\otau$.
 We define 
$ {\bN^{\otau}_{<N}}$ to be $\{\vec{a}\in\bZT\mid 0\leq \vec{a}(e)\leq N-1,e\in\otau  \}$.

Recall that there is an algebraic embedding $F:\Tro\rightarrow \Trq$ defined by $F(\alpha_p) = T_N(\alpha_p),p\in\mathcal{P}, F(x_e) = (x_e)^{N},e\in\otau$.
We use $\Trq^{(N)}$ to denote $\Im F$, then $\Im F=\bC[T_N(\alpha_p)\mid p\in\mathcal{P}][(x_e)^{\pm N}\mid
e\in\otau]$. We have the following obvious Lemma.

\def \Fr {\text{Fr}}

\def \dim {\text{dim}}

\begin{lemma}\label{basis}
Both $\{\alpha^{\vec{a}}x^{\vec{b}}\mid \vec{a}\in \bN^{\mathcal{P}}_{<N},\vec{b}\in \bN^{\otau}_{<N}  \}$
 and $\{{}^{\vec{a}}\alpha x^{\vec{b}}\mid \vec{a}\in \bN^{\mathcal{P}}_{<N},\vec{b}\in \bN^{\otau}_{<N}  \}$ are basis of
$\Trq$ over $\Trq^{(N)}$.
\end{lemma}

We use $\widetilde{\Trq^{(N)}}$ to denote the field of fractions of $\Trq^{(N)}$, and use
 $U$ to denote the set of nonzero elements of $\Trq^{(N)}$. Then $\Trq[U^{-1}]$ is a finite dimension algebra over the field $\widetilde{\Trq^{(N)}}$.

\begin{lemma}\label{quan}
(a)  Both $\{\alpha^{\vec{a}}x^{\vec{b}}\mid \vec{a}\in \bN^{\mathcal{P}}_{<N},\vec{b}\in \bN^{\otau}_{<N}  \}$
 and $\{{}^{\vec{a}}\alpha x^{\vec{b}}\mid \vec{a}\in \bN^{\mathcal{P}}_{<N},\vec{b}\in \bN^{\otau}_{<N}  \}$ are basis of $\Trq[U^{-1}]$ over $\widetilde{\Trq^{(N)}}$. Especially
$$\dim_{\widetilde{\Trq^{(N)}}}
\Trq[U^{-1}] = 3r(\Sigma).$$

(b) $\Trq[U^{-1}]$ is a division algebra over $\widetilde{\Trq^{(N)}}$.

(c)
$\text{Trace}_{\widetilde{\Trq^{(N)}}}$ makes 
$\Trq[U^{-1}]$ into a symmteric Frobenius algebra over $\widetilde{\Trq^{(N)}}$.

(d) $\Trq[U^{-1}] = \Fr(\Trq)$.
\end{lemma}
\begin{proof}
Lemmas \ref{division} and \ref{basis}.
\end{proof}

\def\tcF {\widetilde{\cF}}

\def \tF {\widetilde{F}}

Clearly, the algebraic embedding $\cF:\sSo\rightarrow \sSq$ induces an algebraic embedding
$\widetilde{\cF}:\Fr(\sSo)\rightarrow \sSq[S^{-1}]$ such that $\Im \tcF = \widetilde{\sSq^{(N)}}$.
Similarly, the algebraic embedding $F:\Tro\rightarrow \Trq$ induces an algebraic embedding
$\widetilde{F}:\Fr(\Tro)\rightarrow \Trq[U^{-1}]$ such that 
$\Im \tF = \widetilde{\Trq^{(N)}}$.

\def \wid {\widetilde}

\bl\label{key2}
The commutative diagram in Proposition \ref{commu} induces the following commutative diagram:
\begin{equation}\label{eq_comm}
\begin{tikzcd}
\Fr(\sSo)  \arrow[r, "\tcF"]
\arrow[d, "\wid{\Tr}"]  
&  \sSq[S^{-1}]  \arrow[d, "\wid{\Tr}"] \\
 \Fr(\Tro)  \arrow[r, "\tF"] 
&  \Trq[U^{-1}]\\
\end{tikzcd},
\end{equation}
where the two homomorphisms in the vertical arrows are isomorphisms.
\el
\begin{proof}
Clearly, the commutative diagram in Proposition \ref{commu} induces the commutative diagram in equation \eqref{eq_comm}. Lemma \ref{equal} indicates $\wid{\Tr}:\Fr(\sSo)\rightarrow \Fr(\Tro)$
is an isomorphism. From Lemma \ref{equali} and (d) in \ref{quan}, we have $$\sSq[S^{-1}]
 =
\Fr(\sSq),   \Trq[U^{-1}] = \Fr(\Trq) .$$ Then Lemma \ref{equal} indicates $\wid{\Tr}:\sSq[S^{-1}]\rightarrow \Trq[U^{-1}]$
is also an isomorphism.
\end{proof}

\def \Tr {\text{Trace}}

 $\bC[x]$ is a free $\bC[T_N(x)]$-module with a basis $\{T_0(x), T_1(x),\cdots,T_{N-1}(x)\}$, Proposition 3.2 in \cite{frohman2016frobenius}. For any element $f\in \bC[x]$, we can suppose
$L_f(T_i(x))=fT_i(x) = \sum_{0\leq j\leq N-1} g_{ij} T_j(x)$, where $0\leq i\leq N-1,g_{ij}\in \bC[x]$ for $0\leq i,j\leq N-1$. Then define $\Tr_{\bC[T_N(x)]} (f) = \sum_{0\leq i\leq N-1}g_{ii}$.

\begin{lemma}(\cite{frohman2016frobenius})\label{torus}
For any $k\in\bN$, we have
$$\Tr_{\bC[T_N(x)]} (T_k(x)) = \left\{ 
    \begin{aligned}
    &T_k(x) & & N\mid k \cr 
    &0 & & \text{otherwise.}
    \end{aligned}
\right.$$
\end{lemma}

$\bC[\alpha_p\mid p\in\mathcal{P}]$ is a free $\bC[T_N(\alpha_p)\mid p\in\mathcal{P}]$-module with basis
$\{\alpha^{\vec{a}}\mid \vec{a}\in \bN^{\mathcal{P}}_{<N}\}$. Similarly we can define a $\bC[T_N(\alpha_p)\mid p\in\mathcal{P}]$-linear map
$$\Tr_{\bC[T_N(\alpha_p)\mid p\in\mathcal{P}]}:\bC[\alpha_p\mid p\in\mathcal{P}]\rightarrow\bC[T_N(\alpha_p)\mid p\in\mathcal{P}].$$

\begin{lemma}\label{puncture}
For any $\vec{a}\in \bN^{\mathcal{P}}$, we have
$$\Tr_{\bC[T_N(\alpha_p)\mid p\in\mathcal{P}]} (\alpha^{\vec{a}}) = \left\{ 
    \begin{aligned}
    &\alpha^{\vec{a}} & & N\mid \vec{a}(e)\text{ for all }e\in\otau \cr 
    &0 & & \text{otherwise.}
    \end{aligned}
\right.$$
\end{lemma}
\begin{proof}
Trivially, we have $\bC[\alpha_p\mid p\in\mathcal{P}] = \otimes_{p\in\mathcal{P}}\bC[\alpha_p]$, and 
$\bC[T_N(\alpha_p)\mid p\in\mathcal{P}] = \otimes_{p\in\mathcal{P}}\bC[T_N(\alpha_p)]$. Thus 
$$\Tr_{\bC[T_N(\alpha_p)\mid p\in\mathcal{P}]} (\alpha^{\vec{a}})
 = \prod_{p\in\mathcal{P}}\Tr_{\bC[T_N(\alpha_p)]} (T_{\vec{a}(p)}(\alpha_p)).$$
Then Lemma \ref{torus} completes the proof.
\end{proof}

\begin{theorem}\label{main1} For any $\vec{a}\in \bN^{\mathcal{P}},\vec{b}\in \bN^{\otau}$, we have
\begin{equation}\label{Fr}
\Tr_{\wid{\sSq^{(N)}}}(\alpha ^{\vec{a}} X^{\vec{b}}) = \left\{ 
    \begin{aligned}
    &\alpha ^{\vec{a}} X^{\vec{b}} & & N\mid \vec{a}(p)\text{ for all }p\in\mathcal{P} \text{ and }
N\mid \vec{b}(e)\text{ for all }e\in\otau \cr 
    &0 & & \text{otherwise.}
    \end{aligned}
\right.
\end{equation}
$\Tr_{\wid{\sSq^{(N)}}}:\sSq[S^{-1}]\rightarrow \wid{\sSq^{(N)}}$ as a $\wid{\sSq^{(N)}}$-linear map is unique with the property in equation \eqref{Fr}.
\end{theorem}
\begin{proof}
The uniqueness is obviously, since $\sSq[S^{-1}]$ is linearly spanned by $\{\alpha ^{\vec{a}} X^{\vec{b}}\mid \vec{a}\in \bN^{\mathcal{P}},\vec{b}\in \bN^{\otau}\}$ as a vector space over $ \wid{\sSq^{(N)}}$ ((a) in Lemma \ref{quan} and Lemma \ref{key2}).

Then we try to show equation \eqref{Fr}.
Lemma \ref{key2} indicates it suffices to show
$$\Tr_{\wid{\Trq^{(N)}}}(\alpha ^{\vec{a}} x^{\vec{b}}) = \left\{ 
    \begin{aligned}
    &\alpha ^{\vec{a}} x^{\vec{b}} & & N\mid \vec{a}(p)\text{ for all }p\in\mathcal{P} \text{ and }
N\mid \vec{b}(e)\text{ for all }e\in\otau \cr 
    &0 & & \text{otherwise,}
    \end{aligned}
\right.$$
for any $\vec{a}\in \bN^{\mathcal{P}},\vec{b}\in \bN^{\otau}$.
If $N|\vec{a}(p)\text{ for all }p\in\mathcal{P} \text{ and }
N|\vec{b}(e)\text{ for all }e\in\otau$, we have $\alpha ^{\vec{a}} x^{\vec{b}}\in \Trq^{(N)}$, then 
$$\Tr_{\wid{\Trq^{(N)}}}(\alpha ^{\vec{a}} x^{\vec{b}}) =\alpha ^{\vec{a}} x^{\vec{b}} \Tr_{\wid{\Trq^{(N)}}}(1) = \alpha ^{\vec{a}} x^{\vec{b}}.$$

Suppose there exists $p_0\in \mathcal{P}$ such that $N\nmid \vec{a}(p_0)$ or there exists $e_0\in \otau$ such that $N\nmid \vec{b}(e_0)$.
Recall that
$\{\alpha^{\vec{f}}x^{\vec{g}}\mid \vec{f}\in \bN^{\mathcal{P}}_{<N},\vec{g}\in \bN^{\otau}_{<N}  \}$
 is a  basis of $\Trq[U^{-1}]$ over $\widetilde{\Trq^{(N)}}$.
If there exists $e_0\in \otau$ such that $N\nmid \vec{b}(e_0)$. For any $\vec{c}\in \bN^{\mathcal{P}}_{<N},\vec{d}\in \bN^{\otau}_{<N}$,
$\alpha ^{\vec{a}} x^{\vec{b}}( \alpha ^{\vec{c}} x^{\vec{d}} )= q^{\frac{1}{2}t} \alpha ^{\vec{a}}\alpha ^{\vec{c}} x^{\vec{b}+\vec{d}}$, where $t\in\bZ$. Suppose $\vec{b} + \vec{d} =N \vec{u} +\vec{r}$ where $\vec{u}\in \bN^{\otau},\vec{r}\in \bN^{\otau}_{<N}$. 
Then $\alpha ^{\vec{a}} x^{\vec{b}}( \alpha ^{\vec{c}} x^{\vec{d}} )= q^{\frac{1}{2}t} x^{N\vec{u}} \alpha ^{\vec{a}}\alpha ^{\vec{c}} x^{\vec{r}}.$
Since $N\nmid \vec{b}(e_0)$, we have $\vec{r}(e_0)\neq \vec{d}(e_0)$. Thus, if we write 
$$\alpha ^{\vec{a}} x^{\vec{b}}( \alpha ^{\vec{c}} x^{\vec{d}} )= \sum_{\vec{f}\in \bN^{\mathcal{P}}_{<N},\vec{g}\in \bN^{\otau}_{<N}}u_{\vec{f},\vec{g}} \alpha^{\vec{f}} x^{\vec{g}},$$
where $u_{\vec{f},\vec{g}}\in\wid{\sSq^{(N)}}$ for $\vec{f}\in \bN^{\mathcal{P}}_{<N},\vec{g}\in \bN^{\otau}_{<N}$, we have $u_{\vec{c},\vec{d}} = 0$. Therefore, we have $\Tr_{\wid{\Trq^{(N)}}}(\alpha ^{\vec{a}} x^{\vec{b}})=0.$  If $N\mid \vec{b}(e)\text{ for all }e\in\otau $, then $x^{\vec{b}}\in\Trq^{(N)}$ and 
there exists $p_0\in \mathcal{P}$ such that $N\nmid \vec{a}(p_0)$. 
We have $$\Tr_{\wid{\Trq^{(N)}}}(\alpha ^{\vec{a}} x^{\vec{b}})=
x^{\vec{b}} \Tr_{\wid{\Trq^{(N)}}}(\alpha ^{\vec{a}}) =N^{|\otau|}  x^{\vec{b}} \Tr_{\bC[T_N(\alpha_p)\mid p\in\mathcal{P}]}(\alpha ^{\vec{a}})=
0,$$
where the last equality is from Lemma \ref{puncture}.
\end{proof}

\begin{corollary}\label{666666}
For any $\vec{a}\in \bN^{\mathcal{P}},\vec{b}\in \bN^{\otau}$, we have
\begin{equation}
\Tr_{\wid{\sSq^{(N)}}}(\alpha ^{\vec{a}} X^{\vec{b}}) = \left\{ 
    \begin{aligned}
    &\alpha ^{\vec{a}} X^{\vec{b}} & &\alpha ^{\vec{a}} X^{\vec{b}}\in \sSN  \cr 
    &0 & & \text{otherwise.}
    \end{aligned}
\right.
\end{equation}

\end{corollary}
\begin{proof}
If $\alpha ^{\vec{a}} X^{\vec{b}}\in \sSN$, we have 
$$\Tr_{\wid{\sSq^{(N)}}}(\alpha ^{\vec{a}} X^{\vec{b}}) =\alpha ^{\vec{a}} X^{\vec{b}} \Tr_{\wid{\sSq^{(N)}}}(1)=\alpha ^{\vec{a}} X^{\vec{b}}.$$

If $\alpha ^{\vec{a}} X^{\vec{b}}\notin \sSN$, there exists $p_0\in \mathcal{P}$ such that $N\nmid \vec{a}(p_0)$ or there exists $e_0\in \otau$ such that $N\nmid \vec{b}(e_0)$.
Thus  
Theorem \ref{main1} indicates $\Tr_{\wid{\sSq^{(N)}}}(\alpha ^{\vec{a}} X^{\vec{b}}) =0$.

\end{proof}

\begin{remark}

When $\partial\Sigma=\emptyset$,
a parallel result for Corollary \ref{666666} was proved in \cite{abdiel2017localized,frohman2016frobenius}.

\end{remark}

\begin{corollary}\label{basis1}(a)
Both $\{\alpha^{\vec{a}}X^{\vec{b}}\mid \vec{a}\in \bN^{\mathcal{P}}_{<N},\vec{b}\in \bN^{\otau}_{<N}  \}$
 and $\{{}^{\vec{a}}\alpha X^{\vec{b}}\mid \vec{a}\in \bN^{\mathcal{P}}_{<N},\vec{b}\in \bN^{\otau}_{<N}  \}$ are basis of $\sSq[S^{-1}]$ over $\widetilde{\sSq^{(N)}}$. Especially
$$\dim_{\widetilde{\sSq^{(N)}}}
\sSq[S^{-1}] = 3r(\Sigma).$$

(b) $$\Tr_{\wid{\sSq^{(N)}}}(\sum_{\vec{a}\in \bN^{\mathcal{P}}_{<N},\vec{b}\in \bN^{\otau}_{<N}}u_{\vec{a},\vec{b}} \alpha^{\vec{a}} X^{\vec{b}}) = u_{\vec{0},\vec{0}},$$
where $u_{\vec{a},\vec{b}}\in\wid{\sSq^{(N)}}$. 
\end{corollary}
\begin{proof}
(a) is indicated by Lemmas \ref{quan} and \ref{key2}.

(b) follows from Theorem \ref{main1}.

%
\end{proof}

\begin{remark}
We already calculated 
the dimension in (a) in Theorem \ref{basis1} using a different technique, and  provided a different basis of $\sSq[S^{-1}]$ over $\widetilde{\sSq^{(N)}}$ \cite{wang2023finiteness}.
\end{remark}

\begin{remark}
When every component of the  pb surface has at least one boundary component,
the Frobenius map for  stated $SL_n$-skein algebras was constructed in \cite{wang2023stated}. The quantum trace map for stated $SL_n$-skein algebras was built in \cite{le2023quantum11}. We expect that we can prove  parallel results, as Theorem \ref{main1}, Corollaries \ref{666666} and \ref{basis1},  for stated $SL_n$-skein algebras, using the same 
  techniques used in this subsection.

\end{remark}

%

\def \lam {{\Lambda}}

\subsection{On $\T_{\tSZ}$}
Recall that $\Sigma$ is obtained from a compact surface $\overline{\Sigma} $ by removing finite points, which are called puntured.
We use ${\Lambda}$ to denote the set of boundary components of $\overline{\Sigma}$ that contain even number of punctures. 
For each $c\in\overline{\Sigma}$, we select one boundary component $c(e)$ of $\Sigma$ such that $c(e)\subset c$.

 The orientation of $\partial\Sigma$ induced from $\Sigma$ is called the positive orientation. 
For every $c\in\lam$,
suppose $c$ contains $t$ boundary ideal arcs.
 We consecutively label all these boundary ideal arcs as $e_1,\cdots,e_t$,  using the positive orientation of $\partial \Sigma$ and starting with $c(e)$ (that is $e_1=c(e)$).
We will use $c(e)_{i}$ to denote $e_i$ for $1\leq i\leq t$.
  Then for any integer $0\leq k\leq N-1$, define
$$X_{c,k} = \left\{ 
    \begin{aligned}
    &1 & & k=0 \cr 
    &(X_{\overline{e_1}})^k (X_{\overline{e_2}})^{N-k} \cdots (X_{\overline{e_{t-1}}})^{k} (X_{\overline{e_t}})^{N-k} & & 1\leq k\leq N-1 .
    \end{aligned}
\right.$$

\bt(\cite{yu2023center})\label{center}
As a subalgebra of $\sSq$,
$Z_{q^{1/2}}(\Sigma)$ is generated by $$\sSN,\{\alpha_p\mid p\in\mathcal{P}\},\{X_{c,k}\mid
c\in\lam,0\leq k\leq N-1\}.$$
\et

\begin{corollary}
As a $\sSN$-module,
$Z_{q^{1/2}}(\Sigma)$ is  freely generated by   
$$\{{}\alpha^{\vec{a}}\prod_{c\in\lam}X_{c,k_{c}}\mid \vec{a}\in\bN^{\mathcal{P}}_{< N},0\leq k_c\leq N-1,c\in\lam \}.$$
\end{corollary}
\begin{proof}
From Theorem \ref{center}, we know, as a $\sSN$-module,
$Z_{q^{1/2}}(\Sigma)$ is  generated by   
$\{{}\alpha^{\vec{a}}\prod_{c\in\lam}X_{c,k_{c}}\mid \vec{a}\in\bN^{\mathcal{P}}_{< N},0\leq k_c\leq N-1,c\in\lam \}.$ We also have  $\{{}\alpha^{\vec{a}}\prod_{c\in\lam}X_{c,k_{c}}\mid \vec{a}\in\bN^{\mathcal{P}}_{< N},0\leq k_c\leq N-1,c\in\lam \}$ is linearly independent over $\sSN$, since it is a subset of a basis in (a) in Corollary \ref{basis1}.
\end{proof}

\def \T {\text{Tr}}
\def \Zq{\mathcal{Z}_{q^{1/2},\tau}(\Sigma)}
\def \Zp{\mathcal{Z}_{q^{1/2},\tau}^{+}(\Sigma)}
\def \Zsq {Z_{q^{1/2}}(\Sigma)}

\def \tT {\wid{\T}}

For every $c\in\Lambda$ and $0\leq k\leq N-1$, we use $x_{c,k_c}$ to denote $\T(X_{c,k})\in \Trp$.
Then $x_{c,k}$ lies in the center of $\Trq$ since $\Trp\subset \T(\sSq)\subset \Trq$. We use 
$\Zp$ to denote the subalgebra of $\Trp$ generated by 
$$\mathbb{C}[\alpha_p\mid p\in\mathcal{P}],\{x_e^{N}\mid e\in\otau\},\{x_{c,k}\mid c\in\Lambda,0\leq k\leq N-1\},$$ then $\Zp$ lies in the center of $\Trp$.
Similarly, we use $\Zq$ to denote the subalgebra of $\Trq$ generated by 
$$\mathbb{C}[\alpha_p\mid p\in\mathcal{P}],\{x_e^{\pm N}\mid e\in\otau\},\{(x_{c,k})^{\pm 1}\mid c\in\Lambda,0\leq k\leq N-1\},$$ then $\Zq$ lies in the center of $\Trq$.

We have \begin{equation}\label{888}\Zp\subset \T( Z_{q^{1/2}}(\Sigma))\subset \Zq.\end{equation}
Clearly, we  have the field of fractions of $\Zp$ is the same with the field of fractions of $\Zq$.
Recall that  $Q=\SZ\setminus\{0\}$, and $\tSZ = \SZ[Q^{-1}]$.
We use $V$  to denote $\Zq\setminus\{0\}$. 
Then equation \eqref{888} indicates the embedding $\T:\Zsq\rightarrow \Zq$
 induces an isomorphism  
\begin{equation}
\widetilde{\T}:\tSZ\rightarrow \Zq[V^{-1}].
\end{equation}

\begin{lemma}\label{6688}
$\{x^{\vec{k}}\mid \vec{k}\in\bN^{\otau}_{<N},\vec{k}(c(e)) = 0,c\in\Lambda\}$ is a basis of $\Trq$ over $\Zq$.
\end{lemma}

To prove the above Lemma, we give a different presentation for $\Zq$.
For every $c\in\Lambda$, we use $t_c$ to denote the number of boundary ideal arcs contained in $c$.
Define 
\begin{align*}
\mathcal{A} =\{\vec{a}\in \mathbb{Z}^{\otau}\mid \text{ for every } c\in\Lambda\text{, } \vec{a}(\overline{c(e)_1})
 = k_c,\vec{a}(\overline{c(e)_2})
 = -k_c,\cdots, \vec{a}(\overline{c(e)_{t_c-1}}) = k_c,\\\vec{a}( \overline{c(e)_{t_c}}) = -k_c\text{, where }
0\leq k_c\leq N-1\text{, and set all other entries of } \vec{a}\text{ to be zero}\}.
\end{align*}
Define
\begin{equation}\label{B}
\mathcal{B} = \{\vec{a}\in\mathbb{Z}^{\otau} \mid \vec{a} =\vec{b} +\vec{c} \text{, where }\vec{b}\in
\mathcal{A} \text{,  } \vec{c} \in\mathbb{Z}^{\otau}\text{ and } N\mid \vec{c}(e) \text{ for all } e\in\otau\}.
\end{equation}

\begin{lemma}\label{group}
(a) Define an element $\vec{a}\in\mathbb{Z}^{\otau}$ as following:
for every $c\in\Lambda$, set $\vec{a}(\overline{c(e)_1}) = k_c,
\vec{a}(\overline{c(e)_2}) = -k_c,\cdots,\vec{a}(\overline{c(e)_{t_c-1}}) = k_c,\vec{a}(\overline{c(e)_{t_c}}) =- k_c$, where
$k_c\in\mathbb{Z}$, and set all other entries of $\vec{a}$ to be zero.
Then $\vec{a}\in \mathcal{B}.$

(b) $\mathcal{B}$ is a subgroup of $\mathbb{Z}^{\otau}$.

(c) As a $\mathbb{C}[\alpha_p\mid p\in\mathcal{P}]$-submodule of $\Trq$, $\Zq$ is spanned by 
$\{x^{\vec{a}}\mid \vec{a}\in\mathcal{B}\}$.

\end{lemma}
\begin{proof}
(a) For every $c\in\Lambda$, suppose $k_c = q_cN+r_c$ where $q_c,r_c\in \mathbb{Z}$ and $0\leq r_c\leq N-1$. For every $c\in\Lambda$,
set $\vec{q}(\overline{c(e)_1}) = q_c,
\vec{q}(\overline{c(e)_2}) =- q_c,\cdots,\vec{q}(\overline{c(e)_{t_c-1}}) = q_c,\vec{q}(\overline{c(e)_{t_c}}) =- q_c$ and set all other entries of $\vec{q}$ to be zero. Similarly, we define $\vec{r}$. Then 
$\vec{a} = N\vec{q} + \vec{r}$ where $\vec{r}\in\mathcal{A}$. Thus $\vec{a}\in\mathcal{B}$.

(b) is implied by (a).

(c) 
 is implied by (a) and (b).
\end{proof}


\begin{proof} for Lemma \ref{6688}.
Obviously, $\{x^{\vec{k}}\mid \vec{k}\in\bN^{\otau}_{<N},\vec{k}(c(e)) = 0,c\in\Lambda\}$ is a spanning set of $\Trq$ over $\Zq$.

 We want to show $\{x^{\vec{k}}\mid \vec{k}\in\bN^{\otau}_{<N},\vec{k}(c(e)) = 0,c\in\Lambda\}$ is linearly independent over $\Zq$. From (c) in Lemma \ref{group}, it suffices to show, for 
any $\vec{a},\vec{b}\in\mathcal{B}$, $\vec{c},\vec{d}\in\bN^{\otau}_{<N}$ with $\vec{c}(c(e))=\vec{d}(c(e)) = 0$ for all $c\in\Lambda$, $\vec{a}+\vec{c} = \vec{b} +\vec{d}$ indicates $\vec{a} = \vec{b}$ and $\vec{c} = \vec{d}$.
Since $\vec{a}\in\mathcal{B}$, we can suppose $\vec{a} = \vec{a}_1+\vec{a}_2$, where $\vec{a}_1\in\mathcal{A}$
and $N\mid \vec{a}_2(e)$ for all $e\in\otau$. Similarly, we suppose $\vec{b} = \vec{b}_1+\vec{b}_2$.
For every $c\in\Lambda$, we have $\vec{a}(c(e)) + \vec{c}(c(e)) = \vec{b}(c(e)) +\vec{d}(c(e))$. Since 
$\vec{c}(c(e))=\vec{d}(c(e)) = 0$, we have $\vec{a}(c(e))  = \vec{b}(c(e)) $. Thus 
$ \vec{a}_1(c(e))+\vec{a}_2(c(e)) =  \vec{b}_1(c(e))+\vec{b}_2(c(e))$. Then we get 
$ \vec{a}_1(c(e)) =  \vec{b}_1(c(e)),\vec{a}_2(c(e)) = \vec{b}_2(c(e))$ because 
$N\mid \vec{a}_2(c(e)),N\mid \vec{b}_2(c(e))$  and $0\leq \vec{a}_1(c(e)),\vec{b}_1(c(e))\leq N-1$. From the definition of 
$\mathcal{A}$, we have $\vec{a_1} = \vec{b_1}$. Thus we have $\vec{a}_2+\vec{c} = \vec{b}_2+\vec{d}$. Since 
$N\mid \vec{a}_2(e),N\mid \vec{b}_2(e)$ for all $e\in\otau$, and $0\leq \vec{c}(e),\vec{d}(e)\leq N-1$ for all
$e\in\otau$, then we have $\vec{a}_2= \vec{b}_2$ and $\vec{c} = \vec{d}$.
\end{proof}

\begin{lemma}\label{basis66}
(a)  $\{x^{\vec{k}}\mid \vec{k}\in\bN^{\otau}_{<N},\vec{k}(c(e)) = 0,c\in\Lambda\}$ is a basis of $\Trq[V^{-1}]$ over $\Zq[V^{-1}]$. Especially
$$\dim_{\Zq[V^{-1}]} \Trq[V^{-1}]
= 3r(\Sigma) - |\Lambda| - |\mathcal{P}|.$$

(b) $\Trq[V^{-1}]$ is a division algebra over $\Zq[V^{-1}]$.

(c)
$\text{Trace}_{\Zq[V^{-1}]}$ makes 
$\Trq[V^{-1}]$ into a symmteric Frobenius algebra over $\Zq[V^{-1}]$.

(d) $\Trq[V^{-1}] = \Fr(\Trq)$.
\end{lemma}
\begin{proof}
Lemmas \ref{division} and \ref{6688}.
\end{proof}

\bl\label{key6}
We have the following
commutative diagram:
\begin{equation}\label{eq_com}
\begin{tikzcd}
\tSZ  \arrow[r, ""]
\arrow[d, "\tT"]  
&  \sSq[Q^{-1}]  \arrow[d, "\tT"] \\
 \Zq[V^{-1}]  \arrow[r, ""] 
&  \Trq[V^{-1}]\\
\end{tikzcd},
\end{equation}
where  the two homomorphisms in the horizontal arrows are embeddings, and  the two homomorphisms in the vertical arrows are isomorphisms induced by $\T$.
\begin{proof}
The proof is similar with Lemma \ref{key2}.
\end{proof}
\el

\begin{theorem}\label{main2} For any $\vec{a}\in \bN^{\mathcal{P}},\vec{b}\in \bN^{\otau}$, we have
\begin{equation}\label{Fr6}
\Tr_{\tSZ}(\alpha ^{\vec{a}} X^{\vec{b}}) = \left\{ 
    \begin{aligned}
    &\alpha ^{\vec{a}} X^{\vec{b}} & & \vec{b}\in \mathcal{B} \cr 
    &0 & & \text{otherwise.}
    \end{aligned}
\right.
\end{equation}
$\Tr_{\tSZ}:\sSq[Q^{-1}]\rightarrow \tSZ$ as a $\tSZ$-linear map is unique with the property in equation \eqref{Fr6}.
\end{theorem}
\begin{proof}
The uniqueness is obviously, since $\sSq[Q^{-1}]$ is linearly spanned by $\{\alpha ^{\vec{a}} X^{\vec{b}}\mid \vec{a}\in \bN^{\mathcal{P}},\vec{b}\in \bN^{\otau}$ as a vector space over $\tSZ$.

Then we try to show equation \eqref{Fr6}.
Lemma \ref{key6} indicates it suffices to show
$$\Tr_{\Zq[V^{-1}]}(\alpha ^{\vec{a}} x^{\vec{b}}) = \left\{ 
    \begin{aligned}
    &\alpha ^{\vec{a}} x^{\vec{b}} & & \vec{b} \in\mathcal{B} \cr 
    &0 & & \text{otherwise,}
    \end{aligned}
\right.$$
for any $\vec{a}\in \bN^{\mathcal{P}},\vec{b}\in \bN^{\otau}$.

If $\vec{b}\in \mathcal{B}$, we have $\alpha ^{\vec{a}} x^{\vec{b}}\in \mathcal{Z}_{q^{1/2},\tau}(\Sigma)$. Then 
$\Tr_{\Zq[V^{-1}]}(\alpha ^{\vec{a}} x^{\vec{b}}) = \alpha ^{\vec{a}} x^{\vec{b}}.$

Suppose $\vec{b}\notin\mathcal{B}$. Since 
$\Tr_{\Zq[V^{-1}]}(\alpha ^{\vec{a}} x^{\vec{b}}) =\alpha ^{\vec{a}}  \Tr_{\Zq[V^{-1}]}(x^{\vec{b}})$, it suffices to show $ \Tr_{\Zq[V^{-1}]}(x^{\vec{b}}) = 0 $.
Define an element $\vec{u}\in\mathbb{N}^{\otau}$ as following:
for every $c\in\Lambda$, set $\vec{u}(\overline{c(e)_1}) = \vec{b}(c(e)),
\vec{u}(\overline{c(e)_2}) = -\vec{b}(c(e)),\cdots,\vec{u}(\overline{c(e)_{t_c-1}}) = \vec{b}(c(e)),\vec{u}(\overline{c(e)_{t_c}}) =- \vec{b}(c(e))$,  and set all other entries of $\vec{u}$ to be zero.
From Lemma \ref{group}, we know $\vec{u}\in\mathcal{B}$. Since $\vec{b}\notin\mathcal{B}$ and $\vec{u}\in\mathcal{B}$, then $\vec{b} -\vec{u}\notin\mathcal{B}$. Suppose $\vec{b} - \vec{u} = N\vec{v} +\vec{w}$
where $\vec{v}\in\mathbb{Z}^{\otau},\vec{w}\in\mathbb{N}^{\otau}_{<N}.$ Then $\vec{w}(c(e))
 = 0$, for every $c\in\Lambda$, because $\vec{b}(c(e)) - \vec{u}(c(e)) = 0.$
We have $\vec{w}\notin\mathcal{B}$ because $\vec{b} - \vec{u}\notin\mathcal{B},N\vec{v}\in\mathcal{B}$, and $\mathcal{B}$ is a group.
 We have 
$x^{\vec{b}} = x^{\vec{u}+  N\vec{v} +\vec{w}} = x^{ N\vec{v} +\vec{u}} x^{\vec{w}}$ because $x^{ N\vec{v} +\vec{u}}$ lies in the center of $\Trq$. Since $\vec{u}\in\mathcal{B}$ and $\mathcal{B}$ is a group, we have 
$N\vec{v} +\vec{u}\in\mathcal{B}$. Then $$ \Tr_{\Zq[V^{-1}]}(x^{\vec{b}}) =
x^{ N\vec{v} +\vec{u}} \Tr_{\Zq[V^{-1}]}( x^{\vec{w}}).$$
We use the basis in (a) in Lemma \ref{basis66} to calculate $\Tr_{\Zq[V^{-1}]}( x^{\vec{w}})$. 
Since $\vec{w}\notin\mathcal{B}$, we have $\vec{w}\neq \vec{0}$.
For any $ \vec{k}\in\bN^{\otau}_{<N}$ with $\vec{k}(c(e)) = 0,c\in\Lambda$, 
$x^{\vec{w}}x^{\vec{k}} = q^{\frac{1}{2}n}x^{\vec{w}+\vec{k}}$ where $n\in\mathbb{Z}$.
Suppose
$\vec{w}+\vec{k} = N\vec{c} +\vec{d}$
where $\vec{c}\in\mathbb{Z}^{\otau},\vec{d}\in\mathbb{N}^{\otau}_{<N}.$
 Then $\vec{d}(c(e))
 = 0$, for every $c\in\Lambda$, because $\vec{w}(c(e)) + \vec{k}(c(e)) = 0.$
Thus $x^{\vec{d}}$ is an element in the basis in (a) in Lemma \ref{basis66}.
Since $\vec{w}\neq \vec{0}$, then $\vec{k}\neq{\vec{d}}$.
$x^{\vec{w}}x^{\vec{k}} = q^{\frac{1}{2}n}x^{\vec{w}+\vec{k}} = q^{\frac{1}{2}n}x^{N\vec{c}}x^{\vec{d}}$,
where $q^{\frac{1}{2}n}x^{N\vec{c}}\in \Zq[V^{-1}]$ and $x^{\vec{d}}$ is a basis element that is not equal to 
$x^{\vec{k}}$. Then $\Tr_{\Zq[V^{-1}]} (x^{\vec{w}}) = 0$.
\end{proof}

\begin{corollary}\label{88888888}
For any $\vec{a}\in \bN^{\mathcal{P}},\vec{b}\in \bN^{\otau}$, we have
\begin{equation}
\Tr_{\tSZ}(\alpha ^{\vec{a}} X^{\vec{b}}) = \left\{ 
    \begin{aligned}
    &\alpha ^{\vec{a}} X^{\vec{b}} & &\alpha ^{\vec{a}} X^{\vec{b}}\in \SZ  \cr 
    &0 & & \text{otherwise.}
    \end{aligned}
\right.
\end{equation}

\end{corollary}
\begin{proof}
The proof is similar with Corollary \ref{666666}.
\end{proof}

\begin{remark}
When $\partial\Sigma=\emptyset$,
a parallel result for Corollary \ref{88888888} was proved in \cite{frohman2021dimension}.

\end{remark}

\begin{corollary}\label{basis111}(a)
$\{X^{\vec{k}}\mid \vec{k}\in\bN^{\otau}_{<N},\vec{k}(c(e)) = 0,c\in\Lambda\}$ is a basis of $\sSq[V^{-1}]$ over $\tSZ$. Especially
$$\dim_{\tSZ}\sSq[V^{-1}]
= 3r(\Sigma) - |\Lambda| - |\mathcal{P}|.$$

(b) $$\Tr_{\tSZ}(\sum_{\vec{k}\in\bN^{\otau}_{<N},\vec{k}(c(e)) = 0,c\in\Lambda}u_{\vec{k}}  X^{\vec{k}}) = u_{\vec{0}},$$
where $u_{\vec{k}}\in\wid{\SZ}$. 
\end{corollary}
\begin{proof}
(a) is indicated by Lemmas \ref{basis66} and \ref{key6}.

(b) followes from Theorem \ref{main2}.

\end{proof}

\begin{remark}
Yu calculated the dimension in (a) in Corollary \ref{basis111} for all roots unity \cite{yu2023center}. But he did not give an explicit basis.

\end{remark}

\begin{remark}
We expect that we can prove  parallel results, as Theorem \ref{main2}, Corollaries \ref{88888888} and \ref{basis111},  for all roots unity, using the same 
  techniques used in this subsection.

\end{remark}

\def \ON {\Oq^{(N)}}
\def \tO {\wid{\ON}}
\def \OJ {\Oq[J^{-1}]}

\section{The special cases: the Monogon and the Bigon}\label{6}
Since the stated skein  algebra of a monogon is $\bC$, then everything is trivial with this case.
The stated skein algebra of a bigon 
is $\Oq$ \cite{{costantino2022stated1}}, and the corresponding Frobenius map $\cF:\O\rightarrow\Oq$ is the one in equation \eqref{eq68}.
We use $\ON$ to denote $\Im\cF$, and $J$ to denote $\ON\setminus\{0\}$.
Then $\ON[J^{-1}]$ is a field, which will be denoted as $\tO$, and $\Oq[J^{-1}]$ is a vector space over $\tO$.

\begin{lemma}(\cite{wang2023finiteness})\label{final}
$\{d^{k_1}b^{k_2}c^{k_3}\mid k_1,k_2,k_3\in\mathbb{N}, 0\leq k_1,k_2,k_3\leq N-1\}$ is a basis of $\OJ$ over
$\tO$.
\end{lemma}

\begin{corollary}
For any $k_1,k_2,k_3\in\mathbb{N}$, we have 
\begin{equation}\label{Frr}
\Tr_{\tO}(d^{k_1}b^{k_2}c^{k_3}) = \left\{ 
    \begin{aligned}
    &d^{k_1}b^{k_2}c^{k_3} & & N\mid k_i\text{ for }i=1,2,3\cr 
    &0 & & \text{otherwise.}
    \end{aligned}
\right.
\end{equation}
$\Tr_{\wid{\ON}}:\OJ\rightarrow \wid{\ON}$ as a $\wid{\ON}$-linear map is unique with the property in equation \eqref{Frr}.
\end{corollary}
\begin{proof}
From Lemma \ref{final}, the uniqueness of $\Tr_{\wid{\ON}}$ is trivial. If $N\mid k_i\text{ for }i=1,2,3$, 
$d^{k_1}b^{k_2}c^{k_3}\in\tO$, thus $\Tr_{\tO}(d^{k_1}b^{k_2}c^{k_3}) = d^{k_1}b^{k_2}c^{k_3}$.
Suppose there exists $i_0\in\{1,2,3\}$ such that $N\nmid k_{i_0}$. For any 
$ t_1,t_2,t_3\in\mathbb{N}, 0\leq t_1,t_2,t_3\leq N-1$, we have
$(d^{k_1}b^{k_2}c^{k_3}) (d^{t_1}b^{t_2}c^{t_3}) =q^n d^{k_1+t_1}b^{k_2+t_2}c^{k_3+t_3}$, where $n\in\bZ$.
Since $N\nmid k_{i_0}$, then $k_{i_0}+t_{i_0}\neq t_{i_0} $(mod$N$). Thus 
$\Tr_{\tO}(d^{k_1}b^{k_2}c^{k_3}) = 0$.
\end{proof}

\begin{remark}
It is not true that 
$\Tr_{\tO}(a^{k_1}d^{k_2}b^{k_3}c^{k_4}) =0$ if there exists $i_0\in\{1,2,3,4\}$ such that $N\nmid k_{i_0}$.
For example 
\begin{align*}
&\Tr_{\tO}(adb^{N-1}c^{N-1}) =  \Tr_{\tO}((1+q^{-2}bc)b^{N-1}c^{N-1})\\
=  &\Tr_{\tO}(b^{N-1}c^{N-1}) + q^{-2} \Tr_{\tO}(b^{N}c^{N}) =q^{-2} b^Nc^N.
\end{align*}
\end{remark}

\begin{corollary}
For any $k_1,k_2,k_3\in\mathbb{N}$, we have 
\begin{equation}
\Tr_{\tO}(d^{k_1}b^{k_2}c^{k_3}) = \left\{ 
    \begin{aligned}
    &d^{k_1}b^{k_2}c^{k_3} & & d^{k_1}b^{k_2}c^{k_3}\in\ON\cr 
    &0 & & \text{otherwise.}
    \end{aligned}
\right.
\end{equation}
\end{corollary}

\begin{corollary}
$$\Tr_{\tO}(\sum_{0\leq  k_1,k_2,k_3\leq N-1}u_{k_1,k_2,k_3} d^{k_1}b^{k_2}c^{k_3}) = u_{0,0,0},$$
where $u_{k_1,k_2,k_3}\in \tO,0\leq  k_1,k_2,k_3\leq N-1$.
\end{corollary}

\def \ZO {Z_{q}(SL_2)}

We use $\ZO$ to denote the center of $\Oq$.
For $0\leq i\leq N-1$, we define $x_i = 1$ if $i=0$ and $x_i = b^i c^{N-i}$ if $1\leq i\leq N-1$.
\begin{lemma}\label{lemma}
$\ZO$ is freely generated by $\{x_0,x_1,\cdots,x_{N-1}\}$ over $\ON$.
\end{lemma}

Let $K=\ZO\setminus \{0\}$. We use  $\wid{\ZO}$ to denote $\ZO[K^{-1}]$, then $\Oq[K^{-1}]$ is a vector space over field $\wid{\ZO}$ with dimension $N^2$.

\def \mod{\text{mod}}

\begin{corollary}
$\{d^{k_1}b^{k_2}  \mid k_1,k_2\in\bN,0\leq k_1,k_2\leq N-1\}$ is a basis of $\Oq[K^{-1}]$ over $\wid{\ZO}$.
\end{corollary}
\begin{proof}
The proof is similar with Lemma \ref{6688}.
\end{proof}

\begin{corollary}
For any $k_1,k_2\in\mathbb{N}$, we have 
\begin{equation}\label{F}
\Tr_{\wid{\ZO}}(d^{k_1}b^{k_2}) = \left\{ 
    \begin{aligned}
    &d^{k_1}b^{k_2} & & N\mid k_i\text{ for }i=1,2\cr 
    &0 & & \text{otherwise.}
    \end{aligned}
\right.
\end{equation}
$\Tr_{\wid{\ZO}}:\Oq[K^{-1}]\rightarrow \wid{\ZO}$ as a $\wid{\ZO}$-linear map is unique with the property in equation \eqref{F}.
\end{corollary}

\begin{corollary}
For any $k_1,k_2\in\mathbb{N}$, we have 
\begin{equation}
\Tr_{\wid{\ZO}}(d^{k_1}b^{k_2}) = \left\{ 
    \begin{aligned}
    &d^{k_1}b^{k_2} & & d^{k_1}b^{k_2}\in \ZO\cr 
    &0 & & \text{otherwise.}
    \end{aligned}
\right.
\end{equation}
\end{corollary}

\begin{corollary}
$$\Tr_{\wid{\ZO}}(\sum_{0\leq k_1,k_2\leq N-1}v_{k_1,k_2} d^{k_1}b^{k_2}) = v_{0,0},$$
where $v_{k_1,k_2}\in \wid{\ZO},0\leq k_1,k_2\leq N-1$.
\end{corollary}

\bibliographystyle{plain}

\bibliography{ref.bib}

\hspace*{\fill} \\

School of Physical and Mathematical Sciences, Nanyang Technological University, 21 Nanyang Link Singapore 637371

$\emph{Email address}$: zhihao003@e.ntu.edu.sg

\vspace{0.5cm}

Bernoulli Institute, University of Groningen, P.O. Box 407, 9700 AK Groningen, The Netherlands

$\emph{Email address}$: wang.zhihao@rug.nl

\end{document}

{\cred Try to calculate the dimension of $\Sq$ over $\SqN$ when $M$ is a handle body, more ambitious goal is when $M$ is the thickening of the surface (Techniques used in \cite{frohman2021dimension} maybe helpful). Try and study skein module over image of frobenius for lens space and periodic mapping Tori over torus, also when $M$ is the complement of two bridge knot or two bridge link.  $M$ is the thickening of the closed torus.

 Try to calculation the skein module of trivial  $S^1$-bundles over surfaces (the surface can have boundary) over the ring $\mathbb{C}[A^{\pm 1}]$ (maybe by using techniques in \cite{detcherry2021basis})

Maybe also try to  consider stated case. I think stated case is easier to handle with.

Theorems 2.2 and 2.3 in "Fundamentals of Kauffman bracket skein modules".

MULTIPLICATIVE STRUCTURE OF KAUFFMAN BRACKET
SKEIN MODULE QUANTIZATIONS

ALGEBRAIC GENERATORS OF THE SKEIN ALGEBRA OF A SURFACE

On the genus two skein algebra

On skein algebras of planar surfaces

Proposition 16 in REPRESENTATIONS OF THE KAUFFMAN SKEIN ALGEBRA OF SMALL SURFACES

The skein module of torus knots complements

SOME RESULTS ABOUT THE KAUFFMAN BRACKET SKEIN
MODULE OF THE TWIST KNOT EXTERIOR
}